\documentclass[12pt]{amsart}
\usepackage{amssymb}
\usepackage{amsthm}
\usepackage{amsmath,accents,amsfonts,amssymb,amsthm,mathrsfs,xypic,extarrows,mathtools}
\usepackage{leftidx}
\usepackage{amscd}
\usepackage{extarrows}
\usepackage{graphicx}
\usepackage{fancyhdr}
\usepackage{epsfig}
\usepackage{amsfonts}
\usepackage{mathrsfs}
\usepackage{picinpar} \usepackage{amsmath} \usepackage[all]{xy}
\usepackage[colorlinks=true]{hyperref}
\hypersetup{urlcolor=blue, citecolor=red}
\theoremstyle{plain}
\newtheorem{theorem}{Theorem}[section]
\newtheorem{corollary}[theorem]{Corollary}
\newtheorem{lemma}[theorem]{Lemma}
\newtheorem{proposition}[theorem]{Proposition}
\theoremstyle{definition}

\newtheorem{example}[theorem]{Example}

\newcommand{\Ext}{\mbox{\rm Ext}}
\newcommand{\Hom}{\mbox{\rm Hom}}

\newcommand{\Ker}{\mbox{\rm Ker}}

\newcommand{\Z}{\mbox{\rm Z}}

\begin{document}
\title{Special precovers and preenvelopes of complexes*}
\author{\normalsize {Zhanping WANG$^{1,2}$ \ \ \ \ Zhongkui LIU$^{2}$}\\
\small 1. Department
of Mathematics, Shanghai Jiao Tong University\\
\small Shanghai, 200240, PR China\\
\small 2. Department
of Mathematics, Northwest Normal University\\
\small Lanzhou, 730070, PR China\\
\small E-mail: wangzp@nwnu.edu.cn; liuzk@nwnu.edu.cn\\}
\footnote[0]{*Supported by National Natural Science Foundation of China (Grant No. 11201377, 11261050) and Program of Science and Technique of Gansu Province (Grant No. 1208RJZA145).}
\date{} \maketitle
\hspace{6.3cm}\noindent{\footnotesize {\bf Abstract}}

\vspace{0.2cm} The notion of an $\mathcal{L}$ complex (for a given class of $R$-modules $\mathcal{L}$) was introduced by Gillespie: a complex $C$ is called $\mathcal{L}$ complex
 if $C$ is exact and  $\Z_{i}(C)$ is in $\mathcal{L}$ for all $i\in \mathbb{Z}$. Let $\widetilde{\mathcal{L}}$ stand for the
class of all $\mathcal{L}$ complexes. In this paper, we give sufficient condition on a class of $R$-modules such that every complex has a special $\widetilde{\mathcal{L}}$-precover (resp., $\widetilde{\mathcal{L}}$-preenvelope). As applications, we obtain that  every complex has a special projective precover and a special injective preenvelope, over a coherent ring every complex has a special FP-injective preenvelope, and over a noetherian ring every complex has a special $\widetilde{\mathcal{GI}}$-preenvelope, where $\mathcal{GI}$ denotes the class of Gorenstein injective modules.



\section{Introduction}
Covers and envelopes over any class of modules were defined by Enochs in \cite{EJ81} which unified all the well known covers and envelopes, such as injective envelopes, projective covers and so on. For a given class $\mathcal{L}$ of modules, one of the most important facts about $\mathcal{L}$-(pre)covers and $\mathcal{L}$-(pre)envelopes is that
their existence permits the construction of $\mathcal{L}$-resolutions in an adequate manner to compute homology and cohomology (see \cite{EJ00} for details). This makes very interesting
the study of the existence of $\mathcal{L}$-(pre)covers and $\mathcal{L}$-(pre)envelopes not just when the
class $\mathcal{L}$ is that of all projective or injective modules, but for other important classes
of modules, and not just in the setting of the categories of modules, but for more
general abelian categories. As a particular and important example, the study of
(pre)covers and (pre)envelopes in the category of complexes of modules has been
treated by different authors (see \cite{AEG01, EG99, EG98, EO02, Gar99, Gil04, Ia11, LN11, WL11}). For example, Aldrich et al. proved in \cite{AEG01} that over any ring $R$ every complex admits a flat cover. Also Gillespie \cite{Gil04} showed this by using the
method more analogous to the case of $R$-modules.

Let $\mathcal{L}$ be a class of $R$-modules. According to \cite{Gil04}, a complex $C$ is called $\mathcal{L}$ complex if it is exact and $\Z_{i}(C)$ is in $\mathcal{L}$ for $i\in \mathbb{Z}$, and the class of $\mathcal{L}$ complexes is denoted by $\widetilde{\mathcal{L}}$. For example, if $\mathcal{L}$ is the class of flat (resp., projective, injective) $R$-modules, an $\mathcal{L}$ complex is actually a flat (resp., projective, injective) complex (see \cite{EG98, Gar99}).

In the paper, we give sufficient condition on a class of $R$-modules $\mathcal{L}$ in order for the class of complexes $\widetilde{\mathcal{L}}$ to be (pre)covering or (pre)enveloping in the category of complexes of $R$-modules. It is proved that if $\mathcal{L}$ satisfy $\Ext^{1}(L, L^{'})=0$ for all $L,L^{'}\in \mathcal{L}$, then every module has an epic $\mathcal{L}$-precover if and only if every complex has an epic $\widetilde{\mathcal{L}}$-precover. Let $\mathcal{L}$ be a projectively resolving and special precovering class in $R$-Mod. We show that every complex has a special $\widetilde{\mathcal{L}}$-precover. Also, the special $\widetilde{\mathcal{L}}$-preenvelope of a complex is discussed. As applications, we obtain that every complex has a special projective precover and a special injective preenvelope, over a coherent ring every complex has a special FP-injective preenvelope, and over a noetherian ring every complex has a special $\widetilde{\mathcal{GI}}$-preenvelope, where $\mathcal{GI}$ denotes the class of Gorenstein injective modules.

\section{Preliminaries}
Let $\mathcal{L}$ be class of objects in an abelian category
$\mathcal{C}$. Let $M$ be an object of $\mathcal{C}$. We recall the
definition introduced in \cite{EJ81}. A morphism $f: L\rightarrow M$
is called an $\mathcal{L}$-precover of $M$ if $L\in \mathcal{L}$ and $\Hom(L^{'}, L)\rightarrow \Hom(L^{'}, M)\rightarrow 0$
is exact for all $L^{'}\in \mathcal{L}$. If, moreover, any $g: L\rightarrow L$ such that $fg=f$ is an automorphism of $L$ then $f: L\rightarrow M$ is called an $\mathcal{L}$-cover of $M$. An $\mathcal{L}$-
preenvelope and an $\mathcal{L}$-envelope of $M$ are defined dually. It is immediate that covers and envelopes, if they exist, are unique up to isomorphism, and that if $\mathcal{L}$ contains
all  projective (injective) objects, then $\mathcal{L}$-(pre)covers ($\mathcal{L}$-(pre)envelopes) are always
surjective (injective). An epimorphism $f: L\rightarrow M$ is called a special precover if $\Ext^{1}(L, \Ker(f))=0$ for all $L\in \mathcal{L}$. We say a class $\mathcal{L}$ of objects of $\mathcal{C}$ is (pre)covering if every
object of $\mathcal{C}$ has an $\mathcal{L}$-(pre)covering. Dually, we have the concepts of special preenvelope and (pre)enveloping class.

Throughout this paper, let $R$ be an associative ring, $R$-Mod the category of left $R$-modules and $\mathcal{C}(R)$ the category of complexes of left $R$-modules. A complex \begin{center}$\cdots\stackrel{\delta_{2}}\longrightarrow
C_{1}\stackrel{\delta_{1}}\longrightarrow
C_{0}\stackrel{\delta_{0}}\longrightarrow
C_{-1}\stackrel{\delta_{-1}}\longrightarrow\cdots$
\end{center}
of left $R$-modules will be denoted $(C, \delta)$ or $C$. Given a left
$R$-module $M$, we will denote by $D^{n}(M)$ the complex
\begin{center}$\cdots\longrightarrow
0\longrightarrow M\stackrel{id}\longrightarrow
M\longrightarrow0\longrightarrow\cdots$
\end{center}
with the $M$ in
the $n$ and $(n-1)$-th position. Given a complex $(C,\delta^{C})$, $\Z_{n}(C)=\Ker(\delta_{n}^{C})$.

If $X$ and $Y$ are both complexes of $R$-modules, then by a morphism $f : X\rightarrow Y$ of complexes we mean a sequence of $R$-homomorphisms
$f_{n} : X_{n} \rightarrow Y_{n} $ such that $\delta_{n}^{Y}f_{n}=f_{n-1}\delta_{n}^{X}
$ for each $n\in \mathbb{Z}$. Following \cite{Gar99}, $\Hom(X, Y)$ denotes the set of morphisms of complexes from
$X$ to $Y$ and $\Ext^{i}(X, Y)$ ($i\geq1$) are the right derived functors of $\Hom$.

In what follows, we always assume that all classes of $R$-modules are closed under
isomorphisms and contain zero module.

 For unexplained concepts and notations, we
refer the reader to \cite{EJ00}, \cite{Gar99}, \cite{Gobel2006} and
\cite{Xu96}.

\section{Main results}
Let $\mathcal{L}$ be a class of $R$-modules. Recall from Definition 3.3 in \cite{Gil04} that a complex $C$ is $\mathcal{L}$ complex if it is exact and $\Z_{i}(C)$ is in $\mathcal{L}$ for all $i\in\mathbb{Z}$. It is well known that a complex $C$ is injective (resp., projective, flat) if and only if it is exact and $\Z_{i}(C)$ is injective (resp., projective, flat) $R$-modules for all $i\in\mathbb{Z}$, thus injective (resp., projective, flat) complexes are actually $\mathcal{I}$ (resp., $\mathcal{P}$, $\mathcal{F}$) complexes, where $\mathcal{I}$ (resp., $\mathcal{P}$, $\mathcal{F}$) is the class of injective (resp., projective, flat) $R$-modules. We use $\widetilde{\mathcal{L}}$ to denote the class of all $\mathcal{L}$ complexes.

\begin{lemma}\label{lem1} Suppose $\mathcal{L}$  satisfy $\Ext^{1}(L, L^{'})=0$ for all $L,L^{'}\in \mathcal{L}$. Then every $\mathcal{L}$ complex is a direct sum (or direct product) of complexes in the form $D^{i}(L_{i})$ with $L_{i}\in \mathcal{L}$ and $i\in \mathbb{Z}$.
\end{lemma}
\begin{proof}Let $C=\cdots\rightarrow C_{n+1}\rightarrow C_{n}\rightarrow C_{n-1}\rightarrow \cdots$ be an $\mathcal{L}$ complex. By the hypothesis, we have that each exact sequence $0\rightarrow \Z_{i}(C)\rightarrow C_{i}\rightarrow \Z_{i-1}(C)\rightarrow 0$ is split for any $i\in \mathbb{Z}$. This allows us to write $C_{i}=\Z_{i}(C)\bigoplus\Z_{i-1}(C)$. Thus $C$ is the direct sum (or direct product) of the complexes $\cdots\rightarrow 0\rightarrow \Z_{i}(C)\rightarrow \Z_{i}(C)\rightarrow 0\rightarrow\cdots$.
\end{proof}
\begin{lemma}\label{lem2} If $\varphi: L\rightarrow
C$ is an $\widetilde{\mathcal{L}}$-precover in $\mathcal{C}(R)$, then
$\varphi_{n}: L_{n}\rightarrow C_{n}$ is an $\mathcal{L}$-precover in
$R$-Mod for all $n\in \mathbb{Z}$.
\end{lemma}
\begin{proof} Let $G$ be in $\mathcal{L}$ and $f: G\rightarrow C_{n}$ an
$R$-homomorphism. We define a morphism of complexes $D^{n}(f):
D^{n}(G)\rightarrow C$ as following
\begin{center}
$\xymatrix{
  \cdots \ar[r]  & 0\ar[r] \ar[d] & G\ar[r]^{id}\ar[d]^{f} & G\ar[r]\ar[d]^{\delta ^{n}f} &0\ar[r]\ar[d]&\cdots\\
  \cdots \ar[r]  &C_{n+1} \ar[r] &  C_{n} \ar[r] &C_{n-1} \ar[r] &C_{n-2} \ar[r]  &  \cdots         }$
\end{center}
Since $D^{n}(G)$ is in $\widetilde{\mathcal{L}}$,
there is a morphism $h: D^{n}(G)\rightarrow L$ such that
$\varphi h=D^{n}(f)$. So we have a commutative diagram
\begin{center}
$\xymatrix{
                &         G \ar[d]^{f} \ar[dl]_{h^{n}}    \\
  L_{n}  \ar[r]_{\varphi^{ n}} & C_{n}             }$
\end{center}
This means that $\varphi_{ n}: L_{n}\rightarrow C_{n}$ is an
$\mathcal{L}$-precover of $C_{n}$.
\end{proof}
\begin{proposition} Suppose $\mathcal{L}$  satisfy $\Ext^{1}(L, L^{'})=0$ for all $L,L^{'}\in \mathcal{L}$. Then every module has an epic $\mathcal{L}$-precover if and only if every complex has an epic $\widetilde{\mathcal{L}}$-precover.
\end{proposition}
\begin{proof} $\Rightarrow)$ Suppose $C$ is a complex. By the hypothesis, there exists an epic $\mathcal{L}$-precover $f_{m}: G_{m}\rightarrow C_{m}$ of $C_{m}$ for all $m\in\mathbb{Z}$. Thus there is a morphism $\alpha$ of complexes
\begin{center}
$\xymatrix{
      L=&    \cdots   \ar[r]&G_{m+2}\bigoplus G_{m+1}\ar[d]_{\alpha_{m+1}} \ar[r]^{\delta_{m+1}^{L}}& G_{m+1}\bigoplus G_{m}\ar[d]_{\alpha_{m}} \ar[r]^{\delta_{m}^{L}}& G_{m}\bigoplus G_{m-1}\ar[d]_{\alpha_{m-1}} \ar[r]^{\delta_{m-1}^{L}}& \cdots   \\
  C=& \cdots   \ar[r]&C_{m+1} \ar[r]^{\delta_{m+1}^{C}}&   C_{m} \ar[r]^{\delta_{m}^{C}}& C_{m-1}\ar[r]^{\delta_{m-1}^{C}}& \cdots   }$
\end{center}
where $\delta_{m}^{L}:  G_{m+1}\bigoplus G_{m}\rightarrow G_{m}\bigoplus G_{m-1}$ via $\delta_{m}^{L}(x,y)=(y,0),~~\forall (x,y)\in G_{m+1}\bigoplus G_{m}$, and $\alpha_{m}: L_{m}\rightarrow C_{m}$ via $\alpha_{m}(x,y)=\delta_{m+1}^{C}f_{m+1}(x)+f_{m}(y),~~\forall (x,y)\in G_{m+1}\bigoplus G_{m}$. It is easy to check that $L$ is in $\widetilde{\mathcal{L}}$ and $\alpha$ is epic. It is to prove that $\alpha: L\rightarrow C$ is an $\widetilde{\mathcal{L}}$-precover. Suppose that $Q$ is in $\mathcal{L}$ and $\beta: D^{m}(Q)\rightarrow C$ is a morphism for some $m\in\mathbb{Z}$. We will show that $\alpha_{m}: L_{m}\rightarrow C_{m}$ is an $\mathcal{L}$-precover of $C_{m}$. Then there exists an homomorphism $g_{m}: Q\rightarrow L_{m}$ such that $\alpha_{m}g_{m}=\beta_{m}$. Now define a morphism of complexes $g: D^{m}(Q)\rightarrow L$ via $g_{m-1}=\delta_{m}^{L}g_{m}$ and $g_{i}=0$ for $i\neq m, m-1$. Then $\beta g=\alpha$. That is, the following diagram
\begin{center}
$\xymatrix{
 & D^{m}(Q)\ar[d]\ar[ld]  \\
L \ar[r]  & C}$
\end{center}
commutes. Let $T$ be in $\widetilde{\mathcal{L}}$ and $\gamma: T\rightarrow C$ be a morphism of complexes. By Lemma \ref{lem1}, $T$ is a direct sum (or direct product) of complexes in the form $D^{i}(Q_{i})$ with $Q_{i}\in \mathcal{L}$ and $i\in \mathbb{Z}$.
By the above proof, for any $i\in \mathbb{Z}$ there exists a morphism of complexes $h_{i}: D^{i}(Q_{i})\rightarrow L$ such that $\alpha h_{i}=\gamma\lambda_{i}$. That is, the following diagram
\begin{center}
$\xymatrix{
 & D^{i}(Q_{i})\ar[d]^{\gamma\lambda_{i}}\ar[ld]_{h_{i}}  \\
L \ar[r]_{\alpha}  & C}$
\end{center}
commutes, where $\lambda_{i}: D^{i}(Q_{i})\rightarrow \bigoplus D^{i}(Q_{i})$ is a canonical injection. By the universal property of direct sum, there exists a morphism $\theta: \bigoplus D^{i}(Q_{i})\rightarrow L$ such that $\theta\lambda_{i}=h_{i}$ for any $i\in\mathbb{Z}$. Hence $\alpha\theta=\gamma$, as desired.

$\Leftarrow)$ Suppose $M$ is a module. Then there exists an epic $\widetilde{\mathcal{L}}$-precover $f: L\rightarrow D^{0}(M)$ of $D^{0}(M)$ in $\mathcal{C}(R)$. This implies $f_{0}: L_{0}\rightarrow M$ is an $\mathcal{L}$-precover of $M$ in $R$-Mod by Lemma \ref{lem2}.
\end{proof}

Dual arguments to the above give the following result concerning the monic $\widetilde{\mathcal{L}}$-preenvelopes.

\begin{proposition} Suppose $\mathcal{L}$  satisfy $\Ext^{1}(L, L^{'})=0$ for all $L,L^{'}\in \mathcal{L}$. Then every module has a monic $\mathcal{L}$-preenvelope if and only if every complex has a monic $\widetilde{\mathcal{L}}$-preenvelope.
\end{proposition}

\begin{corollary} Every complex $C$ has an epic projective precover, which is determined by projective precovers of all terms $C_{i}$. Every complex $C$ has a monic injective preenvelope, which is determined by injective preenvelopes of all terms $C_{i}$.
\end{corollary}

\begin{proposition}\label{prop4} If every exact complex has an $\widetilde{\mathcal{L}}$-precover, then any complex has an $\widetilde{\mathcal{L}}$-precover.
\end{proposition}
\begin{proof} Let $C$ be a complex. By \cite[Theorem 3.18]{EJJ96}, we have an exact cover $E\rightarrow C$ of $C$. We take $L\rightarrow E$ to be an $\widetilde{\mathcal{L}}$-precover of $E$. Then, since any complex in $\widetilde{\mathcal{L}}$ is exact, it is easy to see that $L\rightarrow E\rightarrow C$ is an $\widetilde{\mathcal{L}}$-precover of $C$.
\end{proof}
\begin{proposition}\label{prop5} Let $C$ be in $\widetilde{\mathcal{L^{\bot}}}$. If $\mathcal{L}$ is a covering class in $R$-Mod, then $C$ has an $\widetilde{\mathcal{L}}$-cover.
\end{proposition}
\begin{proof}For each $i\in \mathbb{Z}$, we consider the short exact sequence $$0\rightarrow \Z_{i}(C)\rightarrow C_{i}\rightarrow \Z_{i-1}(C)\rightarrow 0.$$ Given $L\rightarrow \Z_{i-1}(C)$ with $L\in \mathcal{L}$, there is a lifting $L\rightarrow C_{i}$ since $\Ext^{1}(L, \Z_{i}(C))=0$. Hence we can construct the diagram
\begin{center}
$\xymatrix{
 0\ar[r] & L_{i}\ar[d]\ar[r] & L_{i}\bigoplus L_{i-1}\ar[d]\ar[r]& L_{i-1}\ar[d]\ar[r]&0 \\
0 \ar[r]  & \Z_{i}(C)\ar[r] & C_{i}\ar[r] & \Z_{i-1}(C)\ar[r]&0}$
\end{center}
where $L_{i}\rightarrow \Z_{i}(C)$ and $L_{i-1}\rightarrow \Z_{i-1}(C)$ are $\mathcal{L}$-covers. It is easy to see that the preceding is an $\widetilde{\mathcal{L}}$-cover. Pasting together all the diagrams for all $ i\in\mathbb{Z}$, we get an $\widetilde{\mathcal{L}}$-cover of $C$.
\end{proof}

Recall that a class of modules is called projectively resolving (injective coresolving) if it is closed under extensions and kernels of surjections (cokernels of injections), and it contains all projective (injective) modules.

\begin{lemma}\label{lem6} Let $\mathcal{L}$ be a projectively resolving class in $R$-Mod and $0\rightarrow A\rightarrow B\rightarrow C\rightarrow 0$ a short exact sequence of $R$-modules. If $L_{1} \stackrel{f_{1}}\rightarrow A $ and $L_{3}\stackrel{f_{3}}\rightarrow C$ are special $\mathcal{L}$-precovers, then there exists a commutative diagram
\begin{center}
$\xymatrix{
& 0\ar[d] & 0\ar[d] & 0\ar[d] \\
0 \ar[r]  & \Ker(f_{1})\ar[r] \ar[d] & \Ker(f_{2})\ar[r]\ar[d] & \Ker(f_{3})\ar[r]\ar[d]&0\\
0\ar[r]  &L_{1} \ar[r]\ar[d]_{f_{1}} & L_{2} \ar[r]\ar[d]_{f_{2}} & L_{3} \ar[r]\ar[d]_{f_{3}}  &0\\
0\ar[r]  &A \ar[r]\ar[d] & B \ar[r] \ar[d]& C \ar[r]\ar[d]  &0\\
& 0 & 0 & 0 }$
\end{center}
with exact rows and columns such that $f_{2}: L_{2} \rightarrow B$ is a special $\mathcal{L}$-precover.
\end{lemma}
\begin{proof}It follows from \cite[Theorem 3.1]{AA02} or \cite[Theorem 3]{DC98}.
\end{proof}

\begin{lemma}\label{lem7} Let $\mathcal{L}$ be a class of $R$-modules. If $G, C$ are exact, $G_{i}, \Z_{i}(G)\in\mathcal{L}$ and $C_{i}, \Z_{i}(C)\in\mathcal{L}^{\perp}$ for all $i\in \mathbb{Z}$, then $\Ext^{1}(G, C)=0.$
\end{lemma}
\begin{proof} Consider the
exact sequence $0\rightarrow C\stackrel{\alpha}\rightarrow
E\stackrel{\beta}\rightarrow D\rightarrow0$ with $E$ an injective
complex. Then $\Hom(G, E)\rightarrow \Hom(G, D)\rightarrow \Ext^{1}(G,
C)\rightarrow0$ is exact. So we only need to prove that $\Hom(G,
E)\rightarrow \Hom(G, D)\rightarrow0$ is exact. Let $f\in \Hom(G, D)$.
Since the sequence $0\rightarrow C\rightarrow E\rightarrow
D\rightarrow0$ is exact and $C$ is exact, we have an exact sequence $0\longrightarrow
\Z_{n-1}(C)\longrightarrow \Z_{n-1}(E)\stackrel{\theta_{n-1}}\longrightarrow \Z_{n-1}(D)\longrightarrow0$ in $R$-Mod for all $n\in \mathbb{Z}$. By
the hypothesis, we have $\Ext^{1}(\Z_{n-1}(G), \Z_{n-1}(C))=0$. It follows that the diagram
\begin{center}
$\xymatrix{
 & & \Z_{n-1}(G) \ar@{.>}[dll]_{h_{n-1}}\ar[d]^{g_{n-1}}\\
    \Z_{n-1}(E)\ar[rr]_{\theta_{n-1}} & &  \Z_{n-1}(D)\ar[rr]&& 0} $
\end{center}
commutes, where $g_{n-1}=f_{n-1}|_{\Z_{n-1}(G)}$. Since
$C$, $E$ are exact, $D$ is also exact. For $G, D, E$, we have the
exact sequences $0\longrightarrow \Z_{n}(G)\stackrel{i_{n}}\longrightarrow
G_{n}\stackrel{p_{n}}\longrightarrow \Z_{n-1}(G)\longrightarrow0$, $0\longrightarrow \Z_{n}(D)\stackrel{l_{n}}\longrightarrow
D_{n}\stackrel{\eta_{n}}\longrightarrow \Z_{n-1}(D)\longrightarrow0$, $0\longrightarrow \Z_{n}(E)\stackrel{e_{n}}\longrightarrow
E_{n}\stackrel{\pi_{n}}\longrightarrow \Z_{n-1}(E)\longrightarrow0$ respectively, and $E_{n}=\Z_{n}(E)\bigoplus \Z_{n-1}(E)$. We define $w_{n}:
G_{n}\rightarrow D_{n}$  as $w_{n}=f_{n}-\beta_{n}\lambda
_{n-1}h_{n-1}p_{n}$ with $\lambda_{n-1}: \Z_{n-1}(E)\rightarrow \Z_{n}(E)\bigoplus \Z_{n-1}(E)$
via $\lambda_{n-1}(x)=(0, x)$. It is easy to check that $\eta_{n}w_{n}=0$.
So there exists $\xi_{n}: G_{n}\rightarrow \Z_{n}(D)$ such
that the following diagram
\begin{center}
$\xymatrix{
  0  \ar[r]^{} & G_{n} \ar@{.>}[d]_{\xi_{n}} \ar[r]^{id} &  G_{n} \ar[d]_{w_{n}} \ar[r]^{} & 0 \ar[d]_{}  \\
  0 \ar[r]^{} &  \Z_{n}(D)\ar[r]_{l^{n}} & D_{n} \ar[r]_{\eta_{n}} & \Z_{n-1}(D) \ar[r]^{} & 0} $
\end{center}
commutes. Since $\Z_{n}(C)$ is in $\mathcal{L}^{\perp}$ and $G_{n}$ is in $\mathcal{L}$, $\Ext^{1}(G_{n},
\Z_{n}(C))=0$, and so there exists $\sigma_{n}:
G_{n}\rightarrow \Z_{n}(E)$ such that
$\theta_{n}\sigma_{n}=\xi_{n}$. This means that there is a
commutative diagram
\begin{center}
$\xymatrix{
 & & & G_{n} \ar@{.>}[dl]_{\sigma_{n}}\ar[d]^{\xi_{n}}\\
 0\ar[r]&\Z_{n}(C) \ar[r] & \Z_{n}(E)\ar[r]_{\theta_{n}}  &  \Z_{n}(D)\ar[r]& 0} $
\end{center}
Now we define $\rho_{n}: G_{n}\rightarrow E_{n}$ via
$\rho_{n}(x)=(\sigma_{n}(x), h_{n-1}p_{n}(x))$, $\forall x\in
G_{n}$, and take $h_{n}=\sigma_{n}i_{n}$. It is not hard to check
that the following diagram
\begin{center}
$\xymatrix{
    & 0 \ar[r]^{}    & \Z_{n}(G) \ar'[d][dd]_{g_{n}} \ar[rr]^{i_{n}}\ar[dl]_{h_{n}} & & G_{n} \ar'[d][dd]_{f_{n}} \ar[rr]^{p_{n}}\ar[dl]_{\rho_{n}} & &\Z_{n-1}(G) \ar[dl]_{h_{n-1}}\ar'[d][dd]_{g_{n-1}} \ar[r]^{} & 0  \\
  0  \ar[r]^{}  &\Z_{n}(E) \ar[dr]_{\theta_{n}} \ar[rr]^{e_{n}}& & E_{n} \ar[dr]_{\beta_{n}} \ar[rr]^{\pi_{n}}& & \Z_{n-1}(E) \ar[dr]_{\theta_{n-1}} \ar[rr]^{} && 0  \\
 & 0  \ar[r]^{} & \Z_{n}(D)  \ar[rr]^{}& & D_{n}  \ar[rr]^{} && \Z_{n-1}(D)  \ar[r]^{}
 &0}$
\end{center}
commutes. On the other hand, since $\Z_{n-1}(E)$ is
injective, there exists an homomorphism $\varphi_{n-1}:
G_{n-1}\rightarrow \Z_{n-1}(E)$ such that the following
diagram
\begin{center}
$\xymatrix{
     0 \ar[r]^{}    & \Z_{n-1}(G) \ar[d]_{h_{n-1}} \ar[rr]^{i_{n-1}}& &G_{n-1}\ar@{.>}[lld]^{\varphi_{n-1}} \\
&\Z_{n-1}(E)}$
\end{center}
commutes. We take $\psi_{n-1}: G_{n-1}\rightarrow \Z_{n-1}(D)$ as $\psi_{n-1}=\theta_{n-1}\varphi_{n-1}$, and consider
the following diagram
\begin{center}
$\xymatrix{
  0  \ar[r]^{} & \Z_{n-1}(G) \ar[d]_{g_{n-1}} \ar[rr]^{i_{n-1}} &&  G_{n-1} \ar[d]_{f_{n-1}} \ar[rr]^{p_{n-1}} \ar[lld]_{\psi_{n-1}}&& \Z_{n-2}(G) \ar[d]^{g_{n-2}}\ar[r]\ar@{.>}[lld]_{t_{n-1}} &0 \\
  0 \ar[r]^{} &  \Z_{n-1}(D)\ar[rr]_{l_{n-1}} && D_{n-1} \ar[rr]_{\eta_{n-1}} && \Z_{n-2}(D) \ar[r]^{} & 0} $
\end{center}
where $g_{n-2}=f_{n-2}|_{\Z_{n-2}(G)}$. Since
$\psi_{n-1}i_{n-1}=g_{n-1}$, there exists an homomorphism $t_{n-1}$
such that $\eta_{n-1}t_{n-1}=g_{n-2}$ and
$t_{n-1}p_{n-1}+l_{n-1}\psi_{n-1}=f_{n-1}$. Since $C_{n-1}$ is
in $\mathcal{L}^{\perp}$ and $\Z_{n-2}(G)$ is in $\mathcal{L}$,
we have a commutative diagram
\begin{center} $\xymatrix{
 & & & \Z_{n-2}(G)\ar@{.>}[dl]_{s_{n-1}}\ar[d]^{t_{n-1}}\\
 0\ar[r]&C_{n-1} \ar[r] & E_{n-1}\ar[r]_{\beta_{n-1}}  & D_{n-1}\ar[r]& 0} $
\end{center}
Now we define $\rho_{n-1}: G_{n-1}\rightarrow E_{n-1}$ as
$\rho_{n-1}=e_{n-1}\varphi_{n-1}+s_{n-1}p_{n-1}$, and take
$h_{n-2}=\pi_{n-1}s_{n-1}$. It is not hard to check that
$\beta_{n-1}\rho_{n-1}=f_{n-1}$, $e_{n-1}h_{n-1}=\rho_{n-1}i_{n-1}$,
$h_{n-2}p_{n-1}=\pi_{n-1}\rho_{n-1}$, and
$\theta_{n-2}h_{n-2}=g_{n-2}$. That is, the following diagram
\begin{center}
$\xymatrix{
    & 0 \ar[r]^{}    & \Z_{n-1}(G) \ar'[d][dd]_{g_{n-1}} \ar[rr]^{i_{n-1}}\ar[dl]_{h_{n-1}} & & G_{n-1} \ar'[d][dd]_{f_{n-1}} \ar[rr]^{p_{n-1}}\ar[dl]_{\rho_{n-1}} & &\Z_{n-2}(G) \ar[dl]_{h_{n-2}}\ar'[d][dd]_{g_{n-2}} \ar[r]^{} & 0  \\
  0  \ar[r]^{}  &\Z_{n-1}(E) \ar[dr]_{\theta_{n-1}} \ar[rr]^{e_{n-1}}& & E_{n-1} \ar[dr]_{\beta_{n-1}} \ar[rr]^{\pi_{n-1}}& & \Z_{n-2}(E) \ar[dr]_{\theta_{n-2}} \ar[rr]^{} && 0  \\
 & 0  \ar[r]^{} & \Z_{n-1}(D)  \ar[rr]^{}& & D_{n-1}  \ar[rr]^{} && \Z_{n-2}(D)  \ar[r]^{}
 &0}$
\end{center}
commutes. Repeating above process, we can construct the morphism of
complexes $\rho: G\rightarrow E$ such that the diagram
\begin{center}
$\xymatrix{
 & G \ar@{.>}[dl]_{\rho}\ar[d]^{f}\\
 E\ar[r]_{\beta}&D \ar[r] &  0} $
\end{center}
commutes in $\mathcal{C}(R)$. This means that $\Hom(G, E)\rightarrow
\Hom(G, D)\rightarrow0$ is exact, as desired.
\end{proof}

\begin{proposition} \label{prop8} Let $\mathcal{L}$ be a projectively resolving and special precovering class in $R$-Mod. Then every exact complex has a special $\widetilde{\mathcal{L}}$-precover.
\end{proposition}
\begin{proof} Let $E$ be an exact complex. Then we have short exact sequences $$0\rightarrow \Z_{i}(E)\rightarrow E_{i}\rightarrow \Z_{i-1}(E)\rightarrow 0.$$
By the hypothesis, there exists a special $\mathcal{L}$-precover $f^{'}_{i}: L^{'}_{i}\rightarrow \Z_{i}(E)$ of the module $\Z_{i}(E)$ for each $i\in \mathbb{Z}$. By Lemma \ref{lem6}, there exists a commutative diagram
\begin{center}
$\xymatrix{
& 0\ar[d] & 0\ar[d] & 0\ar[d] \\
0 \ar[r]  & \Ker(f^{'}_{i})\ar[r] \ar[d] & \Ker(f_{i})\ar[r]\ar[d] & \Ker(f^{'}_{i-1})\ar[r]\ar[d]&0\\
0\ar[r]  &L^{'}_{i} \ar[r]\ar[d]_{f^{'}_{i}} & L_{i} \ar[r]\ar[d]_{f_{i}} & L^{'}_{i-1} \ar[r]\ar[d]_{f^{'}_{i-1}}  &0\\
0\ar[r]  &\Z_{i}(E) \ar[r]\ar[d] & E_{i} \ar[r] \ar[d]& \Z_{i-1}(E) \ar[r]\ar[d]  &0\\
& 0 & 0 & 0 }$
\end{center}
with exact rows and columns such that $f_{i}: L_{i} \rightarrow E_{i}$ is a special $\mathcal{L}$-precover. Pasting together all the diagrams for all $i\in \mathbb{Z}$, we obtain that $$L=\cdots\rightarrow L_{i+1}\rightarrow L_{i}\rightarrow L_{i-1}\rightarrow \cdots$$ is in $\widetilde{\mathcal{L}}$ and $$K=\cdots\rightarrow \Ker(f_{i+1})\rightarrow \Ker(f_{i})\rightarrow \Ker(f_{i-1})\rightarrow \cdots$$ is exact. By Lemma \ref{lem7}, $\Ext^{1}(Q, K)=0$ for any $Q\in \widetilde{\mathcal{L}}$. Hence $L\rightarrow C$ is a special $\widetilde{\mathcal{L}}$-precover.
\end{proof}

\begin{theorem}\label{th2} Let $\mathcal{L}$ be a projectively resolving and special precovering class in $R$-Mod. Then every complex has a special $\widetilde{\mathcal{L}}$-precover.
\end{theorem}
\begin{proof} Let $C$ be any complex. Then by \cite[Theorem 2.2.4]{Gar99} there exists an exact sequence $$0\rightarrow K\rightarrow E\rightarrow C\rightarrow 0$$ where $E\rightarrow C$ is a special exact precover of $C$. By Proposition \ref{prop8}, we have an exact sequence $$0\rightarrow G\rightarrow L\rightarrow E\rightarrow 0$$
where $L\rightarrow E$ is a special $\widetilde{\mathcal{L}}$-precover of $E$. Consider the following pullback diagram
\begin{center}
$\xymatrix{
     &  0\ar[d]_{}  & 0 \ar[d]_{}   \\
      & G\ar@{=}[r] \ar[d]&G\ar[d]\\
   0\ar[r]&W \ar[r]^{} \ar[d]& L \ar[d]_{} \ar[r]^{} & C\ar@{=}[d]_{} \ar[r]^{} & 0  \\
    0\ar[r] &    K  \ar[d]\ar[r]^{} & E\ar[d]_{} \ar[r]&C \ar[r]& 0\\
    &  0& 0  &
      }$
\end{center}
Since $K\in \mathcal{E}^{\perp}$ where $\mathcal{E}$ is the class of all exact complexes, we obtain $K\in \widetilde{\mathcal{L}}^{\perp}$. But $G\in \widetilde{\mathcal{L}}^{\perp}$, then $W\in \widetilde{\mathcal{L}}^{\perp}$. This implies $L\rightarrow C$ is a special $\widetilde{\mathcal{L}}$-precover of $C$.
\end{proof}

Dual arguments to the above give the following result concerning the special $\widetilde{\mathcal{L}}$-preenvelopes.

\begin{theorem}\label{th3} Let $\mathcal{L}$ be an injectively coresolving and special preenvelping class in $R$-Mod. Then every complex has a special $\widetilde{\mathcal{L}}$-preenvelope.
\end{theorem}

\begin{corollary} Every complex has a special projective precover and a special injective preenvelope.
\end{corollary}

A pair of classes of modules
($\mathcal{A}$, $\mathcal{B}$) is called a cotorsion pair (or
cotorsion theory) \cite{Salce1979} if
$\mathcal{A}^{\perp}=\mathcal{B}$ and
$^{\perp}\mathcal{B}=\mathcal{A}$, where $\mathcal{A}^{\perp}=\{B\in
R-Mod: \ \Ext^{1}(A, B)=0\; \text {for all} \ A\in
\mathcal{A}\}$, $^{\perp}\mathcal{B}=\{A\in
R-Mod: \ \Ext^{1}(A, B)=0\; \text {for all} \ B\in
\mathcal{B}\}$.
A cotorsion pair ($\mathcal{A}$, $\mathcal{B}$) is called hereditary
if whenever $0\longrightarrow A'\longrightarrow A\longrightarrow
A''\longrightarrow 0$ is exact with $A, A''\in \mathcal{A}$ then
$A'$ is also in $\mathcal{A}$. A cotorsion pair ($\mathcal{A}$,
$\mathcal{B}$) is called complete if every module has a
special $\mathcal{B}$-preenvelope and a special
$\mathcal{A}$-precover.

\begin{corollary} If $(\mathcal{A}, \mathcal{B})$ is a hereditary and complete cotorsion pair in $R$-Mod, then every complex has a special $\widetilde{\mathcal{A}}$-precover, and special $\widetilde{\mathcal{B}}$-preenvelope.
\end{corollary}

\begin{proof} It follows from Theorem \ref{th2}, Theorem \ref{th3}, and \cite[Theorem 2.1.4]{EO02}.
\end{proof}

Recall that an $R$-module $M$ is FP-injective \cite{B70} if $\Ext^{1}
_{R}(F, M)=0$ for every finitely presented $R$-module $F$. An $R$-module $N$ is FP-projective \cite{MD05} if $\Ext^{1}
_{R}(N, M)=0$ for every FP-injective $R$-module $M$. Let $\mathcal{FP}$, $\mathcal{FI}$ denote the class of FP-projective
$R$-modules and FP-injective $R$-modules respectively. If $R$ is a coherent ring, then the cotorsion pair $(\mathcal{FP}, \mathcal{FI})$ is hereditary and complete by \cite[Proposition 3.6]{MD07}. By \cite[Theorem 2.10]{WL2011}, $\widetilde{\mathcal{FI}}$ exactly denotes the class of FP-injective complexes. So we obtain the following result.

\begin{corollary} If $R$ is a coherent ring,  then every complex has a special $\widetilde{\mathcal{FP}}$-precover, and special FP-injective preenvelope.
\end{corollary}

Recall from \cite{EJ1995} that a left $R$-module $M$ is called Gorenstein injective if there is an exact sequence
\[\cdots \longrightarrow I_{1}\longrightarrow I_{0}\longrightarrow I_{-1}\longrightarrow I_{-2}\longrightarrow \cdots\]
of injective left $R$-modules such that $M\cong\ker (I_{0}\longrightarrow I_{-1})$  and $\Hom_{R} (\mathcal{I}, -)$ leaves the sequence exact. Let $\mathcal{GI}$ denote the class of Gorenstetin injective modules. Over noetherian rings, existence of Gorenstein injective preenvelopes for all modules is proved by Enochs and Lopes-Ramos in \cite{EL02}. Krause \cite[Theorem 7.12]{K05} proves a stronger result:

\begin{lemma}If $R$ is a noetherian ring, then every module has a special Gorenstein injective preenvelope.
\end{lemma}

Based on the above lemma and Theorem \ref{th3}, we have the following corollary.

\begin{corollary} If $R$ is a noetherian ring,  then every complex has a special $\widetilde{\mathcal{GI}}$-preenvelope.
\end{corollary}

\begin{example} Let $\mathcal{L}$ be a projectively resolving class of modules.

(1) If $f: L\rightarrow M$ is an $\mathcal{L}$-cover in $R$-Mod, then $D^{n}(f): D^{n}(L)\rightarrow D^{n}(M)$ is an $\widetilde{\mathcal{L}}$-cover in $\mathcal{C}(R)$ for all $n\in \mathbb{Z}$.

(2) If $f: L\rightarrow M$ is an $\mathcal{L}$-cover in $R$-Mod, then the induced morphism
\begin{center}
$\xymatrix{
          \cdots   \ar[r]&0\ar[d] \ar[r]&    L \ar[d] \ar[r]& L\ar[d] \ar[r]& 0\ar[d] \ar[r]& \cdots   \\
   \cdots   \ar[r]&0 \ar[r]&    L  \ar[r]& M\ar[r]& 0 \ar[r]& \cdots   }$
\end{center}
is an $\widetilde{\mathcal{L}}$-cover in $\mathcal{C}(R)$.

(3) Let $0\rightarrow L\rightarrow G\rightarrow M\rightarrow 0$ be a
short exact sequence in $R$-Mod with $L\in \mathcal{L}$. Let $L^{'}\rightarrow M$ be an $\mathcal{L}$-cover of $M$. Then we form the pullback diagram
\begin{center}
$\xymatrix{
&& & 0\ar[d]_{}  & 0 \ar[d]_{}   \\
&&  &  K  \ar[d]\ar@{=}[r]^{} & K\ar[d]_{}  \\
(seq1)& 0\ar[r]&L\ar[r]^{} \ar@{=}[d]& X\ar[d] \ar[r]& L^{'}\ar[d]_{}\ar[r]&0\\
(seq2)& 0 \ar[r]^{} & L  \ar[r]^{} & G\ar[d]_{} \ar[r]^{} &M  \ar[d]_{}\ar[r]^{} & 0  \\
&&&  0& 0  &
      }$
\end{center}
It is easy to check that the morphism $(seq1)\rightarrow (seq2)$ is
an $\widetilde{\mathcal{L}}$-cover in $\mathcal{C}(R)$.

 (4) Let $0\rightarrow
N\rightarrow M\rightarrow L\rightarrow 0$ be a short exact sequence
with $L\in \mathcal{L}$. Let $L^{'}\rightarrow M$ be an $\mathcal{L}$-cover of $M$ in $R$-Mod. Then we from the pullback
diagram
\begin{center}
$\xymatrix{
     && 0\ar[d]_{}  & 0 \ar[d]_{} &  \\
     &&K\ar[d]_{} \ar@{=}[r]^{} & K \ar[d]_{} \\
    (seq3)  & 0\ar[r]& X\ar[d]^{} \ar[r]& L^{'}\ar[d]\ar[r] &L\ar@{=}[d]\ar[r]&0 \\
(seq4)& 0\ar[r]& N \ar[d]\ar[r]& M\ar[d]\ar[r] &L\ar[r]&0 \\
 & & 0& 0  &
      }$
\end{center}
It is not hard to check that the morphism $(seq3)\rightarrow (seq4)$
is an $\widetilde{\mathcal{L}}$-cover in $\mathcal{C}(R)$.
\end{example}

\end{document}